\documentclass[11pt]{amsart}
\usepackage{amsmath,amssymb}
\usepackage[all]{xy}
\usepackage{txfonts}

\numberwithin{equation}{section}
\newtheorem{thm}{Theorem}[section]
\newtheorem{prop}[thm]{Proposition}
\newtheorem{lem}[thm]{Lemma}
\newtheorem{rem}[thm]{Remark}

\newtheorem{dfn}[thm]{Definition}
\newtheorem{prob}[thm]{Problem}
\newtheorem{ex}[thm]{Example}
\numberwithin{equation}{section}
\newcommand{\Z}{\mathbb{Z}}
\newcommand{\Q}{\mathbb{Q}}
\newcommand{\R}{\mathbb{R}}
\newcommand{\C}{\mathbb{C}}

\newcommand{\g}{\mathfrak{g}}
\newcommand{\G}{\mathcal{G}}

\newcommand{\sch}{\mathcal{S}}

\newcommand{\Sp}{S\!p}

\newcommand{\SU}{S\!U}
\newcommand{\rank}{\mathrm{rank}}

\begin{document}

\title{Three presentations of Torus equivariant cohomology of flag manifolds}
\author{Shizuo KAJI}
\date{12, Jan, 2015}
\address{Department of Mathematical Sciences,
Faculty of Science, Yamaguchi University,  
1677-1, Yoshida, Yamaguchi 753-8512, Japan /
School of Mathematics, University of Southampton,
    Southampton SO17 1BJ, United Kingdom}
\thanks{The author was partially supported by KAKENHI, Grant-in-Aid for Young
     Scientists (B) 26800043 and JSPS Postdoctoral Fellowships for Research Abroad.}
\email{skaji@yamaguchi-u.ac.jp}

\dedicatory{Dedicated to Professor S.A. Ilori on the occasion
     of his $70$th birthday}

\begin{abstract}
Let $G$ be a compact connected Lie group and
$T$ be its maximal torus.
The homogeneous space $G/T$ is called the (complete) flag manifold.
One of the main goals of the {\em equivariant Schubert calculus} is to
study the $T$-equivariant cohomology $H^*_T(G/T)$
with regard to the $T$-action on $G/T$ by multiplication.
There are three presentations known for $H^*_T(G/T)$;
(1) the free $H^*(BT)$-module generated by the Schubert varieties
(2) (with the rational coefficients) the {\em double coinvariant ring} of the Weyl group
(3) the {\em GKM ring} associated to the Hasse graph of the Weyl group.
Each presentation has both advantages and disadvantages.

In this paper, we describe how to convert an element in one
presentation to another by giving an explicit algorithm,
which can then be used to compute the equivariant structure constants for the product of Schubert classes.
The algorithm is implemented as a Maple script.
\end{abstract}
\maketitle
%%%%%%%%%%%%%%%%%%%%%%%%%%%%%%%%%%%%%%%%%%%

\section{Introduction}
Let $G$ be a compact connected Lie group and
$T$ be its maximal torus.
The homogeneous space $G/T$ is a smooth projective variety 
called the flag manifold.
Among active studies on the topology and geometry of $G/T$ is 
to understand its cohomology ring $H^*(G/T;\Z)$.
As a module, it is well-known that $H^*(G/T;\Z)$ is 
generated freely over $\Z$ by 
the {\em Schubert classes} $\{X_w\mid w\in W\}$ indexed by the Weyl group $W$ of $G$.
To determine the {\em structure constants} $c_{uv}^w\in \Z$
for the product of two Schubert classes
\[
 X_u \cdot X_v = \sum_{w\in W} c_{uv}^w X_w \quad (u,v \in W)
\]
is one of the central problems in {\em Schubert calculus} (see, for example, \cite{duan,Kleiman,KL}).
One method to compute $c_{uv}^w$ is to use the Borel presentation (\cite{Bor}) of $H^*(G/T;\Q)$,
where elements are represented by polynomials and their product is just the product of
polynomials.
Since $H^*(G/T;\Z)$ is torsion-free, and hence, $H^*(G/T;\Z) \to H^*(G/T;\Q)$ is injective,
one can compute $c_{uv}^w$ as follows:
\begin{enumerate}
\item Find polynomial representatives of $X_u$ and $X_v$ in the Borel presentation.
\item Express the product of the two polynomial as a linear sum of the Schubert classes.
\end{enumerate}
To carry out (1) and (2), we have to know how to
convert an element represented by a linear sum of Schubert classes to a polynomial in the Borel presentation and vice versa.
A conversion method was given by Berstein-Gelfand-Gelfand \cite{BGG} and independently by Demazure \cite{Dem}.

We can consider a similar problem in an equivariant setting.
The flag manifold $G/T$ admits the action of $T$ by multiplication from the left.
The Schubert classes generate the equivariant cohomology $H^*_T(G/T;\Z)$ freely 
as $H^*(BT;\Z)$-module (see \S 2).
The problem of determining the structure constants $c_{uv}^w \in H^*(BT;\Z)$ in
\[
 X_u \cdot X_v = \sum_{w\in W} c_{uv}^w X_w \quad (u,v \in W)
\]
is a generalisation of the one in the ordinary cohomology case:
$c_{uv}^w$ is now a polynomial whose constant term is the structure constant for the ordinary cohomology.
Graham showed \cite{Graham} that $c_{uv}^w$ is a polynomial in
the simple roots with non-negative coefficients.
To compute $c_{uv}^w$, we can use the Borel presentation for $H^*_T(G/T;\Q)$ as in the ordinary cohomology case.
Furthermore, for the equivariant cohomology, we have yet another presentation of $H^*_T(G/T;\Q)$ called the {\em GKM presentation}
\cite{GHZ}, which allows a purely combinatorial treatment.
Again we need a conversion method among those three presentations.
The aim of this paper is to describe an algorithm to convert an element represented in one presentation to another
(\S \ref{algorithm}).
We also implemented our algorithm in the computer algebra system Maple (available at \verb+http://skaji.org/code/+),
and demonstrate it by computing the structure constants (\S \ref{demo}).

\subsection*{Acknowledgements}
The author gratefully acknowledges the referee for his or her helpful comments.

\section{Three presentations of $H^*_T(G/T)$}\label{cohomology}
In this section, we review three presentations of the torus equivariant cohomology of the flag manifold.
We refer the reader to \cite{Ku} for details on the subject.

The left multiplication induces an action of $T$ on $G/T$.
The {\em Borel construction} is the total space of the 
Borel fibration with regard to this action:
\[
 G/T \hookrightarrow ET \times_T G/T \xrightarrow{\,\pi\,} BT,
\]
where $T\hookrightarrow ET \to BT$ is the universal $T$-bundle.
More concretely, 
the Borel construction $ET \times_T G/T$ is the quotient space of $ET\times G/T$
by the equivalence relation 
\[
 (y,gT) \sim (ty,tgT) \quad \text{for } y\in ET, gT \in G/T, t\in T.
\]
The Borel $T$-equivariant cohomology $H^*_T(G/T;\Z)$ is by definition the ordinary cohomology
$H^*(ET\times_T G/T;\Z)$.
It is an algebra over $H^*(BT;\Z)$ through $\pi^*: H^*(BT;\Z) \to H^*(ET\times_T G/T;\Z)$.

The Weyl group $W$ of $G$ is a Coxeter group generated by the simple reflections 
$s_1,\ldots,s_n$, where $n=\rank(G)$. We denote the length of $w\in W$ by $l(w)\in \Z$,
the identity element by $e\in W$, and the longest element by $w_0\in W$.
The group $W$ acts $T$-equivariantly on $G/T$ from the right by
$w \cdot gT = gwT$ since $w$ is represented by a normaliser of $T$.
This induces a left action of $W$ on $H^*_T(G/T;\Z)$.

\subsection{Chevalley presentation}
Let $G_\C$ be the complexification of $G$ and $B$ its Borel subgroup.
By the Iwasawa decomposition, $G_\C/B$ is diffeomorphic to $G/T$.
Let $w_0\in W$ be the longest element and $B^{-} = w_0Bw_0$ be the opposite Borel subgroup to $B$.
The Bruhat decomposition 
\[
 G_\C/B = \bigcup_{w\in W} B^{-}wB/B
\]
is $T$-equivariant, and the closure of the cell $B^{-}wB/B$ defines
the {\em Schubert variety} $X_w$.
Denote by the same symbol $X_w$ the class in $H^*_T(G/T;\Z)$ represented by $X_w$.

\begin{rem}
To be more precise, 
one can use the equivariant Gysin map to the Bott-Samelson resolution \cite{BS}
 or the Borel-Moore homology on an approximation space of the Borel construction \cite[Appendix B]{ful}
 to define the class $X_w$.
\end{rem}
\begin{thm}[\cite{Che}]
 $H^*_T(G/T;\Z)$ is a free $H^*(BT;\Z)$-module generated by the classes $X_w$ indexed by $W$.
 The cohomological degree of $X_w$ is $2l(w)$, twice the length of $w$.  
\end{thm}

\subsection{Borel presentation}
With the rational coefficients, $H^*_T(G/T;\Q)$ has a simple presentation
as the {\em double coinvariant ring} of $W$: 
Let $\Q[z_1,z_2,\ldots,z_n]^W$ be the invariant ring of $W$. Then, it is well-known
\begin{thm}[see, for example, \cite{K2}]
 $H^*_T(G/T;\Q)$ is isomorphic to
 $H^*(BT;\Q) \otimes_{H^*(BT;\Q)^W} H^*(BT;\Q)$ 
 as algebras over $H^*(BT;\Q)$.
 Furthermore, since $H^*(BT;\Q)$ is a polynomial algebra, 
 we have
 \[
 H^*_T(G/T;\Q) \simeq \dfrac{\Q[t_1,t_2,\ldots,t_n,x_1,x_2,\ldots,x_n]}{\left(g(t_1,\ldots,t_n)-g(x_1,\ldots,x_n) 
  \mid g\in \Q[z_1,z_2,\ldots,z_n]^W \right)},
 \] 
 where $n=\rank(G)$ and the $H^*(BT;\Q)$-algebra structure on the right hand side 
 is given by the multiplication of $\Q[t_1,t_2,\ldots,t_n]$.
 We call this quotient ring the double invariant ring of $W$ and 
 denote it by $R_W$.
\end{thm}

 An element of $R_W$ or its representative in the polynomial ring
 $\Q[t_1,t_2,\ldots,t_n,x_1,x_2,\ldots,x_n]$ is called a double polynomial
 and denoted by $f(t;x)$.
 To distinguish the two $H^*(BT;\Q)$ factors,
 we often denote by $g(t)$ an element in the left factor $H^*(BT;\Q)$,
 and by $g(x)$ an element in the right $H^*(BT;\Q)$ factor.

\begin{rem}
The above theorem holds with $\Z$-coefficients when $G=\SU(n+1)$ and $\Sp(n)$
but does not in general (see, for example, \cite{Fuk,FIM,Sato1,Sato2}).
\end{rem}

\subsection{GKM presentation}
Konstant and Kumar \cite{KK} essentially gives a combinatorial description
of $H^*_T(G/T;\Q)$, which is now referred to as the GKM presentation.
Arabia \cite{Ara} states the result clearly and gives an independent proof.
Goresky-Kottwitz-MacPherson \cite{GKM} discusses the torus equivariant cohomology
of a manifold with certain properties extending the case of flag manifolds.
Historically, \cite{KK} comes earlier, however, it is common to call the presentation with 
the initials of Goresky-Kottwitz-MacPherson.

The fixed points of the $T$-action on $G/T$ are $\{ vT/T \mid v\in W \}$.
Let $i_v: vT/T \hookrightarrow G/T$ be the inclusion of the fixed point.
Note that $i_v$ is $T$-equivariant.
Putting them together, we obtain the {\em localization map}
\[
 \bigoplus_{v\in W}i^*_v: H^*_T(G/T;\Q) \to \bigoplus_{v\in W} H^*(BT;\Q),
\]
where we identify $H^*_T(vT/T;\Q)\simeq H^*(BT;\Q)$.
Now, define a graph $\G$ as follows:
The vertices are the elements of $W$
and there is an edge between $u,v\in W$ if and only if $s_\beta v = u$ for some reflection $s_\beta$
with regard to a positive root $\beta$.
Label the edge with the ideal $(\beta)$ of $H^*(BT;\Q)$ generated by $\beta$.
The graph $\G$ is called the GKM-graph and $H^*_T(G/T;\Q)$ is given
a purely combinatorial description:
\begin{thm}[\cite{Ara,GHZ,KK}]
Define 
\[
H^*(\G;\Q) := \left\{ \bigoplus_{w\in W} h_v \in \bigoplus_{v\in W}H^*(BT;\Q) \mid h_v-h_u\in (\beta) \text{ if } s_\beta v = u \right\}
\subset \bigoplus_{v\in W}H^*(BT;\Q).
\]
Regard $H^*(\G;\Q)$ as a $H^*(BT;\Q)$-algebra
by factor-wise multiplication.
Then, $H^*_T(G/T;\Q) \simeq H^*(\G;\Q)$ as $H^*(BT;\Q)$-algebras.
This isomorphism is given by the localization map.
\end{thm}
Note that in this presentation, a single element in the cohomology 
is represented by a series of polynomials indexed by the Weyl group.
We denote an element by $h\in H^*(\G;\Q)$ and its factor corresponding to $v\in W$ by $h_v$.

\begin{rem}
When any two positive roots are relatively prime, that is, there is no integer greater than one which divides 
more than one positive roots, the above theorem still holds if $\Q$ is replaced with $\Z$
(\cite{HHH}).
However, this is not true, for example, for type $C$.
The positive roots of $C_2$ are taken to be $2t_1, 2t_2, t_2\pm t_1$.
Let $X_{w_0}$ be the Schubert class corresponding to the longest element $w_0 \in W$.
Then, 
\[
i_v(X_{w_0}) = \begin{cases} 2t_1 \cdot 2t_2 (t_2+t_1) ( t_2-t_1) & (v = w_0) \\ 0 & (v\neq w_0). \end{cases}
\]
While $\frac12 X_{w_0} \in H^*(\G;\Z)$, this class is not in $H^*_T(G/T;\Z)$.

See \cite[Theorem 11.3.9]{Ku} for an analogous result to the above theorem with the $\Z$-coefficients.
\end{rem}

\subsection{Remark on incomplete flag manifolds}
For an incomplete flag manifold $G/P$, where $P$ is a parabolic subgroup,
the Chevalley and the GKM presentations are valid if the index set $W$ is replaced by the set of cosets $W/W_P$,
where $W_P$ is the Weyl group of $P$.
The Borel presentation is given by
\[
 H^*(BT;\Q) \otimes_{H^*(BT;\Q)^W} H^*(BP;\Q),
\]
where $H^*(BP;\Q)$ is isomorphic to the invariant ring $H^*(BT;\Q)^{W_P}$ of $W_P$.
All the arguments in the rest of this paper work for $H^*_T(G/P)$ with minor modification.

\section{Conversion algorithm}\label{algorithm}
From now on, we work with the rational coefficients.
This is not restrictive since 
$H^*_T(G/T;\Q) \simeq H^*(BT;\Q)\otimes_{H^*(BT;\Z)} H^*_T(G/T;\Z)$ and 
$H^*_T(G/T;\Z) \subset H^*_T(G/T;\Q)$ is a subalgebra.

We describe how to convert an element in one presentation 
to another. Examples are given in \S \ref{demo}.

\subsection{Divided difference operator}\label{divided-difference}
In \cite{BGG,Dem}, a series of operators called the {\em divided difference operators}
on $H^*(G/T;\Q)$ were introduced.
Here, we use its straightforward extension to the equivariant setting,
of which we will make heavy use later.
\begin{dfn}
Let $\alpha_i$ be a simple root and $s_i$ be the corresponding simple reflection.
Let $P_i$ be the centraliser of $\ker(\alpha_i)$, where $\alpha_i$ is regarded as a homomorphism $T\to S^1$. 
We have the following oriented, $T$-equivariant fibre bundle
\[
 P_i/T \hookrightarrow G/T \xrightarrow{\pi} G/P_i,
\]
where $P_i/T \simeq S^2$.
Define 
\[
\Delta_i := \pi^* \circ \pi_*,
\]
where $\pi_*:H^*_T(G/T;\Q) \to H^{*-2}_T(G/P_i;\Q)$
be the Gysin map.

For $w\in W$, define
\[
\Delta_w := \Delta_{i_1}\circ \Delta_{i_2}\circ \cdots \circ \Delta_{i_{l(w)}},
\]
where $w=s_{i_1}s_{i_2}\cdots s_{i_{l(w)}}$ is a reduced word.
\end{dfn}

\begin{thm}[see, for example, \cite{BGG,K2,Ku}]\label{divided-diff}
\begin{enumerate}
\item $\Delta_w$ does not depend on the choice of a reduced word of $w\in W$, and thus, is well-defined.
\item For a Schubert class $X_w$, $\Delta_i$ is computed by
\[
 \Delta_i (X_w ) = \begin{cases} X_{w s_i} & ( l(ws_i)=l(w)-1 ) \\ 0 & (l(ws_i)=l(w)+1) \end{cases}.
\]
\item For an element $f(t;x)\in R_W$ in the Borel presentation, $\Delta_i$ is computed by
\[
 \Delta_i(f(t;x)) = \dfrac{f(t;x)-s_i(f(t;x))}{-\alpha_i(x)},
\]
where the simple reflection $s_i$ acts on the $x$-variables (not on the $t$-variables) in $f(t;x)$.
\item For an element $h\in H^*(\G;\Q)$ in the GKM presentation, $\Delta_i$ is computed by
\[
\left(\Delta_i(h) \right)_v = \dfrac{h_v-h_{vs_i}}{-v(\alpha_i)}
\]
\end{enumerate}
\end{thm}

\subsection{Weyl group action}
Another important ingredients for our algorithm is the Weyl group action on 
$H^*_T(G/T;\Q)$ defined in \S \ref{cohomology}.
We will describe it in all the three presentations.
\begin{prop}
\begin{enumerate}
\item For Schubert classes, 
\[
 s_i X_w  = \begin{cases} 
 X_w - w(\alpha_i)(t) X_{ws_i} - \sum_{\beta} \dfrac{2(\alpha_i,\beta)}{(\beta,\beta)}X_{w s_i s_\beta} & (l(ws_i)=l(w)-1) \\
 X_w & (l(ws_i)=l(w)+1)
 \end{cases},
\]
where $\beta$ runs all those positive roots that $l(w)=l(ws_i s_\beta)$.
\item In the Borel presentation, $W$ acts on the right $H^*(BT;\Q)$ factor in $H^*(BT;\Q)\otimes_{H^*(BT;\Q)^W} H^*(BT;\Q)$
by the standard representation. That is, it acts on $x$-variables.
\item For an element $h\in H^*(\G;\Q)$ in the GKM presentation,
\[
 (w\circ h)_v = h_{vw}.
\]
\end{enumerate}
\end{prop}
\begin{proof}
We only show the action on the Schubert classes since the other two are easy to see.
By Theorem \ref{divided-diff} (1) and (2),
we have $s_i X_w = X_w + \alpha_i(x) X_{ws_i}$.
By the Chevalley formula (see, for example, \cite{K}),
we have
\[
\alpha_i(x) X_{v} = v(\alpha_i)(t) X_v - \sum_{\beta} \dfrac{2(\alpha_i,\beta)}{(\beta,\beta)}X_{v s_\beta},
\]
where $\beta$ runs all those positive roots that $l(w s_\beta)=l(w)+1$.
Combining these two, we obtain the formula for $s_i X_w$.
\end{proof}

\subsection{Borel to GKM presentation}
Now, we are ready to give a conversion method among elements in the three presentations.
We begin with an algorithm to convert from the Borel presentation to the GKM presentation.
\begin{prob}
 Given a double polynomial $f(t;x) \in R_w$ in the Borel presentation, 
 find a series of polynomials $\{ h_v \mid v\in W \} \in H^*(\G;\Q)$
 in the GKM presentation which represents the same class as $f(t;x)$.
\end{prob}
Let $ev: \Q[t_1,t_2,\ldots,t_n,x_1,x_2,\ldots,x_n] \to \Q[t_1,t_2,\ldots,t_n]$
be the evaluation at $x_i=t_i \ (1\le i\le n)$.
We often denote $ev(f(t;x))$ by $f(t;t)$.
It is easy to see that 
$i^*_e(f(t;x))=ev(f(t;x))$ and $v \circ i_e= i_v$.
Therefore, we have
\begin{prop}\label{Borel2GKM}
The series of polynomials
 $\{ ev \circ v (f(t;x)) \mid v\in W \}\in H^*(\G;\Q)$
 represents the same class as $f(t;x)\in R_W$.
\end{prop}

\subsection{Schubert to Borel presentation}
\begin{prob}
 Find a double polynomial $\sch_w \in H^*(BT;\Q)\otimes H^*(BT;\Q)$
 which represents the Schubert class $X_w$.
\end{prob}
The polynomial representative $\sch_w$ is often referred to 
as the {\em double (equivariant) Schubert polynomial}
(see, for example, \cite{ful2,FP,Kre-Tam,LS}).
A uniform formula for any Lie type is given by the author:
\begin{thm}[\cite{K}]\label{Schubert2Borel}
For $w\in W$, define the set of 
{\em $k$-factor decompositions} 
$P_k(w)$ recursively for $1\le k\le l(w)$ by
\begin{align*}
 P_k(e) &= \emptyset \\
 P_k(w) &= \bigcup_{s_i} \biggl(
 \{ (w_1,w_2,\ldots,w_{k-1}, w_k s_i) \mid (w_1,\ldots,w_k) \in P_k(ws_i) \}  \\
 &  {} \hspace{1cm} \cup 
   \{ (w_1,w_2,\ldots,w_{k-1}, s_i) \mid (w_1,\ldots,w_{k-1}) \in P_{k-1}(ws_i) \} 
 \biggr), 
\end{align*}
where the union runs over all simple reflections $s_i$ such that $l(ws_i)=l(w)-1$.

Then, $\sch_e(t;x)=1$ and
\[
\sch_w(t;x)=
\sum_{i=1}^{l(w)}
\sum_{(w_1, w_2, \ldots, w_i) \in P_i(w)} (-1)^{i}
\sigma_{w_1}(t)\sigma_{w_2}(t)\cdots \sigma_{w_{i-1}}(t) (\sigma_{w_i}(t) -\sigma_{w_i}(x)),
\]
where $\sigma_w \in H^*(BT;\Q)$ is any polynomial representative of the ordinary Schubert class in
the ordinary cohomology.
For example, by \cite{BGG} $\sigma_w$ is given recursively by
\[
 \sigma_{w_0} = \dfrac{1}{|W|} \prod_{\beta \in \Phi^+} (-\beta), \quad
 \sigma_w = \Delta_{w^{-1}w_0}\sigma_{w_0}, 
\]
where $\Phi^+$ is the set of the positive roots.
\end{thm}

\subsection{Schubert to GKM presentation}
\begin{prob}
 Find a series of polynomials $\{ h_v \mid v\in W \} \in H^*(\G;\Q)$
 in the GKM presentation which represents the Schubert class $X_w$.
\end{prob}
The answer is given by the famous Billey's formula:
\begin{thm}[\cite{Billey}]\label{Schubert2GKM}
Let $v=s_{i_1}\cdots s_{i_{l(v)}}$ be a reduced word.
Define the set of subwords $Q_v(w)$ which multiply to $w$:
\[
 Q_v(w) := \{ (j_1, j_2, \ldots, j_{l(w)}) \mid s_{i_{j_1}}\cdots s_{i_{j_{l(w)}}}=w \}.
\]
The localization image of the Schubert class is determined to be
\[
h_v = i^*_v(X_w)=\sum_{Q_v(w)} \beta_{j_1}\cdots \beta_{j_{l(w)}}
 \]
where $\beta_{j_k}=s_{i_1}\cdots s_{i_{j_k-1}}\alpha_{j_k}$.
\end{thm}

\subsection{GKM to Schubert presentation}
\begin{prob}
 Given a series of polynomials $h=\{ h_v \mid v\in W \} \in H^*(\G;\Q)$
 in the GKM presentation. 
 Determine the coefficients $d_w\in H^*(BT;\Q)$ in the linear sum $\sum_{w\in W} d_w X_w$
 which represents the same class as $h$.
\end{prob}
First, by Theorem \ref{Schubert2GKM},
we have
\[
 i^*_e(X_w)= \begin{cases} 1 & (w=e) \\ 0 & (w\neq e) \end{cases}.
\]
Combining this with Theorem \ref{divided-diff}, we have
\begin{equation}\label{coeff-of-X_w}
 d_w = i^*_e \left(\Delta_w (\sum_{v\in W} d_v X_v) \right) = (\Delta_w(h))_e,
\end{equation}
and we obtain
\begin{prop}\label{GKM2Schubert}
 A series of polynomials $h\in H^*(\G;\Q)$
 represents the class
 \[
 \sum_{w\in W} (\Delta_w(h))_e X_w.
 \]
\end{prop}

\subsection{Borel to Schubert presentation}
\begin{prob}
 Given a double polynomial $f(t;x)\in R_W$
 in the Borel presentation. 
 Determine the coefficients $d_w\in H^*(BT;\Q)$ in the linear sum $\sum_{w\in W} d_w X_w$
 which represents the same class as $f(t;x)$.
\end{prob}
We can use Equation (\ref{coeff-of-X_w}) again to obtain
\begin{prop}\label{Borel2Schubert}
 A double polynomial $f(t;x)\in R_W$
 represents the class
 \[
 \sum_{w\in W} ev(\Delta_w(f)) X_w.
 \]
\end{prop}
Computationally, the following properties are very useful:
\begin{lem}
\begin{enumerate}
\item $\Delta_w$ is $\Q[t_1,\ldots,t_n]$-linear.
\item $\Delta_i(f\cdot g) = \Delta_i(f)g + s_i(f) \Delta_i(g)$.
\item For a fundamental weight $\omega_j$, 
\[
 s_i(\omega_j) = \begin{cases} \omega_i-\alpha_i & (i=j) \\
 \omega_j & (i\neq j) \end{cases},
\]
and hence,
\[
 \Delta_i \left( g(t;x)\omega_i(x)^m \right)
 = g(t;x)\sum_{k=1}^m (\omega_i(x)-\alpha_i(x))^{k-1}\omega_i(x)^{m-k},
\]
where $g(t;x)$ is a polynomial on $t_j \ (1\le j\le n)$ and 
$\omega_j(x) \ (1\le j \le n, j\neq i)$.
\end{enumerate}
\end{lem}
The proof is straightforward by definition.
The last formula is quite efficient for computing the divided difference operators
on the Borel presentation.
\subsection{GKM to Borel presentation}
\begin{prob}\label{prob:GKM2Borel}
 Given a series of polynomials $h=\{ h_v \mid v\in W \} \in H^*(\G;\Q)$
 in the GKM presentation. 
 Find a double polynomial $f(t;x)\in R_W$
 which represents the same class as $h$.

 By Proposition \ref{Borel2GKM}, this is equivalent to
 the following interpolation problem:
 Find a double polynomial $f(t;x)$ which satisfies
 \[
 ev(v\circ f(t;x)) = h_v.
 \]
\end{prob}
Of course, this is solved 
by first converting $h$ to a linear sum of Schubert classes by
Proposition \ref{GKM2Schubert}, and then
use Theorem \ref{Schubert2Borel} to convert the result to 
the Borel presentation.
However, there is no direct, more efficient method known.
In fact, 
many researches to find double Schubert polynomials
have been done by solving this problem in a particular case of $X_{w_0}$:
By Theorem \ref{Schubert2GKM},
$i^*_v(X_{w_0})$ vanishes unless $v=w_0$
and $i^*_{w_0}(X_{w_0})$ is the product of all positive roots.
Even this seemingly easy case, it turned out to be very difficult in general.
Once a polynomial representative $\sch_{w_0}$ for $X_{w_0}$ is found,
$\sch_{w}$ for $w\neq w_0$ can be defined inductively by Theorem \ref{divided-diff}
to be $\sch_w := \Delta_{w^{-1}w_0}\sch_{w_0}$.

\section{Demonstration of the algorithm}\label{demo}
We give concrete examples for the algorithm introduced in the previous section.

We take $G/T = \SU(3)/T$ as our example.
We fix the notation as follows:
\begin{itemize}
\item The Weyl group $W$ is the symmetric group on $3$-letters and generated by the simple reflections (transpositions) $s_1$ and $s_2$.
\item $W = \{ e, s_1, s_2, s_1s_2, s_2s_1, s_1s_2s_1 \}$ and $w_0 = s_1s_2s_1 = s_2s_1s_2$.
\item With the one-line notation for the symmetric group
\[
s_1=(213), s_2=(132), s_1s_2=(231), s_2s_1=(312), s_1s_2s_1 = (321).
\]
\item The dual Lie algebra $\g^*$ of $G$ is the real two dimensional vector space 
\[
\g^* = \{ (z_1,z_2,z_3) \in \R^3 \mid z_1+z_2+z_3=0 \}.
\]
\item $W$ acts on $\g^*$ by the standard representation: 
\[s_1(z_1)=z_2, s_1(z_2)=z_1, s_1(z_3)=z_3, s_2(z_1)=z_1, s_2(z_2)=z_3, s_2(z_3)=z_2.\]
\item The set of positive roots is
\[
 \Phi^+ = \{ z_j - z_i \mid j>i \},
\]
and the simple roots are $\alpha_1 = z_2-z_1, \alpha_2=z_3-z_2$.
\item $H^*(BT;\Q) \simeq \dfrac{\Q[z_1,z_2,z_3]}{(z_1+z_2+z_3)} \simeq \Q[\alpha_1,\alpha_2]$
and $H^*(BT;\Q)^W \simeq \Q[e_2,e_3]$, where $e_2=z_1z_2+z_2z_3+z_3z_1$ and $e_3=z_1z_2z_3$.
\item The Borel presentation of $H^*_T(\SU(3)/T;\Q)$ is given by
\[
R_W = \dfrac{\Q[t_1,t_2,t_3,x_1,x_2,x_3]}
{ \left( t_1+t_2+t_3 ,x_1+x_2+x_3, (t_1t_2+t_2t_3+t_3t_1) - (x_1x_2+x_2x_3+x_3x_1), t_1t_2t_3-x_1x_2x_3 \right) }
\]
\item The GKM-graph is
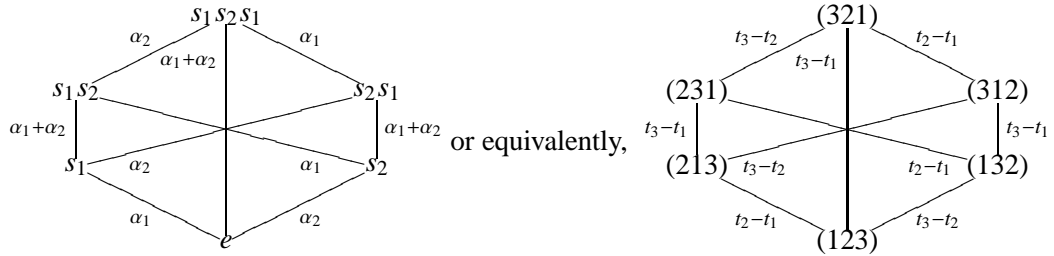
\[
\raisebox{-35pt}{
\begin{xy}
(20,30) *{s_1s_2s_1}="p_0", 
(0,20) *{s_1s_2}="p_1", 
(40,20) *{s_2s_1}="p_2",
(0,10) *{s_1}="p_3", 
(40,10) *{s_2}="p_4",
(20,0) *{e}="p_5", 
\ar@{-} "p_0";"p_1"_{\alpha_2}
\ar@{-} "p_0";"p_2"^{\alpha_1}
\ar@{-} "p_0";"p_5"_(.2){\alpha_1+\alpha_2}
\ar@{-} "p_1";"p_4"_(0.8){\alpha_1}
\ar@{-} "p_1";"p_3"_{\alpha_1+\alpha_2}
\ar@{-} "p_2";"p_3"^(0.8){\alpha_2}
\ar@{-} "p_2";"p_4"^{\alpha_1+\alpha_2}
\ar@{-} "p_3";"p_5"_{\alpha_1}
\ar@{-} "p_4";"p_5"^{\alpha_2}
\end{xy}}
\text{or equivalently,}
\raisebox{-35pt}{
\begin{xy}
(20,30) *{(321)}="p_0", 
(0,20) *{(231)}="p_1", 
(40,20) *{(312)}="p_2",
(0,10) *{(213)}="p_3", 
(40,10) *{(132)}="p_4",
(20,0) *{(123)}="p_5", 
\ar@{-} "p_0";"p_1"_{t_3-t_2}
\ar@{-} "p_0";"p_2"^{t_2-t_1}
\ar@{-} "p_0";"p_5"_(.2){t_3-t_1}
\ar@{-} "p_1";"p_4"_(0.8){t_2-t_1}
\ar@{-} "p_1";"p_3"_{t_3-t_1}
\ar@{-} "p_2";"p_3"^(0.8){t_3-t_2}
\ar@{-} "p_2";"p_4"^{t_3-t_1}
\ar@{-} "p_3";"p_5"_{t_2-t_1}
\ar@{-} "p_4";"p_5"^{t_3-t_2}
\end{xy}
}
\]
\end{itemize}

Let $f(t;x) = t_1 x_1 x_2 \in R_W$.
By Proposition \ref{Borel2GKM}, the corresponding class $h=\{h_v\mid v\in W\}$ is determined 
by substituting
\[
x_i \leftarrow t_{v(i)} \quad (1\le i \le 3).
\]
Therefore, we have
\begin{equation}\label{eq:h}
 h = 
\raisebox{-40pt}{
 \begin{xy}
(20,30) *{t_1t_2t_3}="p_0", 
(0,20) *{t_1t_2t_3}="p_1", 
(40,20) *{t_1^2 t_3}="p_2",
(0,10) *{t_1^2 t_2}="p_3", 
(40,10) *{t_1^2 t_3}="p_4",
(20,0) *{t_1^2 t_2}="p_5", 
\ar@{-} "p_0";"p_1"
\ar@{-} "p_0";"p_2"
\ar@{-} "p_0";"p_5"
\ar@{-} "p_1";"p_4"
\ar@{-} "p_1";"p_3"
\ar@{-} "p_2";"p_3"
\ar@{-} "p_2";"p_4"
\ar@{-} "p_3";"p_5"
\ar@{-} "p_4";"p_5"
\end{xy}
}.
\end{equation}
As we mentioned earlier in Problem \ref{prob:GKM2Borel},
there is no direct algorithm known to convert in the opposite direction.

Next, let us compute $\Delta_{s_1s_2}(h) = \Delta_1\circ \Delta_2(h)$.
By Theorem \ref{divided-diff},
\[
 \Delta_{2}(h)_v = \dfrac{ h_v - h_{v s_2}}{-v(\alpha_2)} 
\]
so 
\[
\Delta_{2}(h)=
\raisebox{-40pt}{
 \begin{xy}
(20,30) *{t_1t_3}="p_0", 
(0,20) *{t_1t_2}="p_1", 
(40,20) *{t_1 t_3}="p_2",
(0,10) *{t_1 t_2}="p_3", 
(40,10) *{t_1^2}="p_4",
(20,0) *{t_1^2}="p_5", 
\ar@{-} "p_0";"p_1"
\ar@{-} "p_0";"p_2"
\ar@{-} "p_0";"p_5"
\ar@{-} "p_1";"p_4"
\ar@{-} "p_1";"p_3"
\ar@{-} "p_2";"p_3"
\ar@{-} "p_2";"p_4"
\ar@{-} "p_3";"p_5"
\ar@{-} "p_4";"p_5"
\end{xy}}, \quad
\Delta_1\circ \Delta_{2}(h)=
\raisebox{-40pt}{
 \begin{xy}
(20,30) *{t_1}="p_0", 
(0,20) *{t_1}="p_1", 
(40,20) *{t_1}="p_2",
(0,10) *{t_1}="p_3", 
(40,10) *{t_1}="p_4",
(20,0) *{t_1}="p_5", 
\ar@{-} "p_0";"p_1"
\ar@{-} "p_0";"p_2"
\ar@{-} "p_0";"p_5"
\ar@{-} "p_1";"p_4"
\ar@{-} "p_1";"p_3"
\ar@{-} "p_2";"p_3"
\ar@{-} "p_2";"p_4"
\ar@{-} "p_3";"p_5"
\ar@{-} "p_4";"p_5"
\end{xy}
}.
\]
Similarly, we can compute $\Delta_1 (h)=0$ and by Proposition \ref{GKM2Schubert}
\[
 \sum_{w\in W} \Delta_w(h)_e X_w = t_1^2t_2 + t_1^2 X_{s_2} + t_1 X_{s_1s_2}
\]
corresponds to the class $h\in H^*(\G;\Q)$.
We, of course, obtain the same presentation by applying Proposition \ref{Borel2Schubert}
directly to $f(t;x)$; it is easy to compute 
\[
\Delta_1(f)=0, \Delta_2(f)=t_1 x_2, \Delta_{s_1s_2}(f)= t_1.
\]
Now, we apply Theorem \ref{Schubert2Borel} to find a double polynomial representative for
\[
t_1^2t_2 + t_1^2 X_{s_2} + t_1 X_{s_1s_2}.
\]
We use the ordinary Schubert polynomial \cite{LS} for $\sigma_w$:
\[
 \sigma_{s_1s_2s_1}(z) = z_1^2z_2, \sigma_{s_1s_2}(z) = z_1z_2,
 \sigma_{s_2s_1}(z) = z_1^2, \sigma_{s_1}=z_1, \sigma_{s_2}=z_1+z_2,
 \sigma_1(z) = 1. 
\]
The set $P_k(w_0)$ can be computed as follows:
\[
P_1(w_0) = w_0,
P_2(w_0) = \{ (s_1s_2, s_1), (s_1, s_2s_1), (s_2s_1, s_2), (s_2, s_1s_2) \},
P_3(w_0) = \{ (s_1,s_2,s_1), (s_2,s_1,s_2) \}.
\]
Then, for example,
\begin{align*}
\sch_{w_0}(t;x) =&
\sum_{i=1}^{l(w)}
\sum_{(w_1, w_2, \ldots, w_i) \in P_i(w_0)} (-1)^{i}
\sigma_{w_1}(t)\sigma_{w_2}(t)\cdots \sigma_{w_{i-1}}(t) (\sigma_{w_i}(t) -\sigma_{w_i}(x)) \\
=& -(\sigma_{w_0}(t)-\sigma_{w_0}(x)) 
+\sigma_{s_1s_2}(t)(\sigma_{s_1}(t)-\sigma_{s_1}(x))
+\sigma_{s_1}(t)(\sigma_{s_2s_1}(t)-\sigma_{s_2s_1}(x)) \\
&
+\sigma_{s_2s_1}(t)(\sigma_{s_2}(t)-\sigma_{s_2}(x))
+\sigma_{s_2}(t)(\sigma_{s_1s_2}(t)-\sigma_{s_1s_2}(x)) \\
&
-\sigma_{s_1}(t)\sigma_{s_2}(t)(\sigma_{s_1}(t)-\sigma_{s_1}(x))
-\sigma_{s_2}(t)\sigma_{s_1}(t)(\sigma_{s_2}(t)-\sigma_{s_2}(x)) \\
=& (x_1-t_1)(x_1-t_2)(x_2-t_1)
\end{align*}
Similarly, we compute
\[
 \sch_{s_2}(t;x) = x_1+x_2-t_1-t_2, \sch_{s_1s_2}(t;x) = (x_1-t_1)(x_2-t_1)
\]
and
\[
t_1^2t_2 + t_1^2 \sch_{s_2} + t_1 \sch_{s_1s_2} = t_1 x_1 x_2 = f(t;x).
\]
In general, the resulting polynomial may differ from $f(t;x)$; they are congruent modulo the ideal.

\begin{rem}
For type $A$ case, 
Theorem \ref{Schubert2Borel} recovers the double Schubert polynomial in \cite{LS} from
the ordinary Schubert polynomial as polynomials (not only modulo the ideal).
\end{rem}

Next, we apply Theorem \ref{Schubert2GKM}
to find the representative $h\in H^*(\G;\Q)$ for
$t_1^2t_2 + t_1^2 X_{s_2} + t_1 X_{s_1s_2}$.
For example, the localization of $X_{s_2}$ at $s_2s_1s_2$ is computed as follows:
\[
 Q_{s_2s_1s_2}(s_2) = \{ (1), (3) \}.
\]
and
\[
 i^*_{s_2s_1s_2}(X_{s_2}) = \alpha_2 + s_2s_1(\alpha_2) = \alpha_1 + \alpha_2 = t_3-t_1.
 \]
Similarly, 
\[
i^*_{s_2s_1s_2}(X_{s_1s_2}) = s_2(\alpha_1)s_2s_1(\alpha_2) = (\alpha_1+\alpha_2)\alpha_1
=(t_3-t_1)(t_2-t_1).
\]
We have
\begin{align*}
 i^*(t_1^2t_2) &=
\raisebox{-40pt}{
  \begin{xy}
(20,30) *{t_1t_2}="p_0", 
(0,20) *{t_1t_2}="p_1", 
(40,20) *{t_1 t_2}="p_2",
(0,10) *{t_1 t_2}="p_3", 
(40,10) *{t_1t_2}="p_4",
(20,0) *{t_1t_2}="p_5", 
\ar@{-} "p_0";"p_1"
\ar@{-} "p_0";"p_2"
\ar@{-} "p_0";"p_5"
\ar@{-} "p_1";"p_4"
\ar@{-} "p_1";"p_3"
\ar@{-} "p_2";"p_3"
\ar@{-} "p_2";"p_4"
\ar@{-} "p_3";"p_5"
\ar@{-} "p_4";"p_5"
\end{xy}}, \quad
 i^*(X_{s_2}) =
\raisebox{-40pt}{
  \begin{xy}
(20,30) *{t_3-t_1}="p_0", 
(0,20) *{t_3-t_1}="p_1", 
(40,20) *{t_3-t_2}="p_2",
(0,10) *{0}="p_3", 
(40,10) *{t_3-t_2}="p_4",
(20,0) *{0}="p_5", 
\ar@{-} "p_0";"p_1"
\ar@{-} "p_0";"p_2"
\ar@{-} "p_0";"p_5"
\ar@{-} "p_1";"p_4"
\ar@{-} "p_1";"p_3"
\ar@{-} "p_2";"p_3"
\ar@{-} "p_2";"p_4"
\ar@{-} "p_3";"p_5"
\ar@{-} "p_4";"p_5"
\end{xy}}, \\
 i^*(X_{s_1s_2}) &=
\raisebox{-40pt}{
  \begin{xy}
(20,30) *{(t_2-t_1)(t_3-t_1)}="p_0", 
(0,20) *{(t_2-t_1)(t_3-t_1)}="p_1", 
(40,20) *{0}="p_2",
(0,10) *{0}="p_3", 
(40,10) *{0}="p_4",
(20,0) *{0}="p_5", 
\ar@{-} "p_0";"p_1"
\ar@{-} "p_0";"p_2"
\ar@{-} "p_0";"p_5"
\ar@{-} "p_1";"p_4"
\ar@{-} "p_1";"p_3"
\ar@{-} "p_2";"p_3"
\ar@{-} "p_2";"p_4"
\ar@{-} "p_3";"p_5"
\ar@{-} "p_4";"p_5"
\end{xy}}.
\end{align*}
Therefore, 
we have
\[
i^*(t_1^2t_2 + t_1^2 X_{s_2} + t_1 X_{s_1s_2})=
\raisebox{-40pt}{
\begin{xy}
(20,30) *{t_1t_2t_3}="p_0", 
(0,20) *{t_1t_2t_3}="p_1", 
(40,20) *{t_1^2 t_3}="p_2",
(0,10) *{t_1^2 t_2}="p_3", 
(40,10) *{t_1^2 t_3}="p_4",
(20,0) *{t_1^2 t_2}="p_5", 
\ar@{-} "p_0";"p_1"
\ar@{-} "p_0";"p_2"
\ar@{-} "p_0";"p_5"
\ar@{-} "p_1";"p_4"
\ar@{-} "p_1";"p_3"
\ar@{-} "p_2";"p_3"
\ar@{-} "p_2";"p_4"
\ar@{-} "p_3";"p_5"
\ar@{-} "p_4";"p_5"
\end{xy}},
\]
which coincides with Equation (\ref{eq:h}).

\section{Structure constants}
As we mentioned in Introduction,
we can compute the equivariant structure constants
through the conversion among three presentations.
Theoretically, to compute the structure constants in
\[
 X_u X_v = \sum_{w\in W}c_{uv}^w X_w,
\]
we can first convert the left hand side to either 
the GKM or the Borel presentation to
compute the product, and then, convert the result back to a linear sum of Schubert classes.
However, this scheme is not practical,
for example, for $G=E_8$ and $\dim(E_8/T)=2l(w_0)=240$.

Here, we describe how we can actually compute the structure constants
at least when $l(u)+l(v)$ is not large.

\subsection{Using the GKM presentation}
This method is introduced in \cite{Billey}.
By Theorem \ref{Schubert2GKM}, we have the following upper triangularity:
\[
i^*_{p} (X_w) = 0 \text{ unless } p\ge w.
\]
Hence, to compute $c_{uv}^w$ we have only to compute
\[
 i^*_{p}(X_u), i^*_{p}(X_v), i^*_p(X_q), i^*_w(X_u), i^*_w(X_v), i^*_w(X_w),
\]
for those $p,q\in W$ that $p,q<w$.
Then, we can determine $c_{uv}^q$ inductively on $q$
and
\[
 c_{uv}^w = i^*_{w}(X_u)i^*_w(X_v) - \sum_{q<w} c_{uv}^q i^*_w(X_q).
\]

We give another method which is much faster in many cases.

\subsection{Using the Borel presentation}
Since Proposition \ref{Borel2Schubert} is efficiently
computed, the bottle-neck is Theorem \ref{Schubert2Borel}.
To find $\sch_w(t;x)$, we only need $\sigma_v$ for $v\le w$.
However, if we use
\cite{BGG} described in Theorem \ref{Schubert2Borel},
we have to compute $\sigma_v$ for all $v\in W$ from the top class $\sigma_{w_0}$.
We give an alternative way to avoid this. 
The idea is very simple; 
the monomials in $\Q[x_1,\ldots,x_n]$ 
generate $H^*(G/T;\Q)\simeq \dfrac{\Q[x_1,\ldots,x_n]}{\Q^+[x_1,\ldots,x_n]^W}$.
Therefore, by solving the following linear system
\[
 \left\{ x_J = \sum_{l(v)=k} \Delta_v(x_J) \sigma_v \mid 
 x_J\in \Q[x_1,\ldots,x_n] \text{ is a monomial of degree } 2k 
 \right\}
\]
we obtain polynomial representatives $\sigma_v$ for all $v\in W$ such that $l(v)=k$.
Once $\sigma_w$ is found with this method, 
we can use $\sigma_v = \Delta_{v^{-1}w}\sigma_w$ for $v \le w$
and apply Theorem \ref{Schubert2Borel} to obtain $\sch_w(t;x)$.

\begin{ex}
Let $G=E_8$ the exceptional Lie group of rank $8$ with the Dynkin diagram: \\
 \setlength{\unitlength}{1mm}
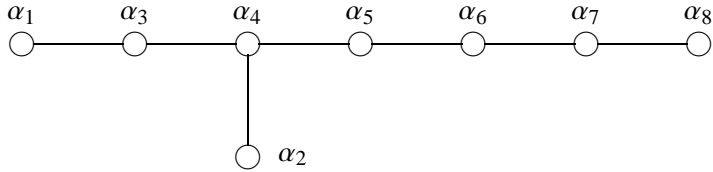
     \begin{picture}(100,30)
        \multiput(15,20)(15,0){7}{\circle{3}}     
        \put(45,5){\circle{3}}     
        \multiput(16.5,20)(15,0){6}{\line(1,0){12}}
        \put(45,18.5){\line(0,-1){12}}     
        
        \put(15,25){\makebox(0,0)[t]{$\alpha_{1}$}}    
        \put(30,25){\makebox(0,0)[t]{$\alpha_{3}$}}
        \put(45,25){\makebox(0,0)[t]{$\alpha_{4}$}}
        \put(60,25){\makebox(0,0)[t]{$\alpha_{5}$}}
        \put(75,25){\makebox(0,0)[t]{$\alpha_{6}$}}
        \put(90,25){\makebox(0,0)[t]{$\alpha_{7}$}}
        \put(105, 25){\makebox(0,0)[t]{$\alpha_{8}$}}
        \put(53,5){\makebox(0,0)[r]{$\alpha_{2}$}}    
   \end{picture}
   
We determine the product $X_{s_4s_2}^2 \in H^*_T(E_8/T;\Q)$.
First, compute polynomial representatives $\sigma_v$ in the ordinary cohomology
for the subwords of $s_4s_2$, that is, $v=s_4s_2, s_4$ and $s_2$:
\begin{align*}
\sigma_{s_2} &= -5\alpha_1(z)-8\alpha_2-10\alpha_3-15\alpha_4-12\alpha_5-9\alpha_6-6\alpha_7-3\alpha_8 \\
\sigma_{s_4} &= -10\alpha_1-15\alpha_2-20\alpha_3-30\alpha_4-24\alpha_5-18\alpha_6-12\alpha_7-6\alpha_8 \\
\sigma_{s_4s_2} &= \sigma_{s_2}^2
\end{align*}
We can compute $\sch_{s_4s_2}(t;x)$ by Theorem \ref{Schubert2Borel}
\[
 \sch_{s_4s_2}(t;x) = -(\sigma_{s_4s_2}(t)-\sigma_{s_4s_2}(x)) + \sigma_{s_4}(t)(\sigma_{s_2}(t)-\sigma_{s_2}(x)) 
\]
 and we have
\begin{align*}
 X_{s_4s_2}^2 =& \sum_{v\le s_4s_2} \Delta_v(\sch_{s_4s_2}^2)(t;t) X_v
 = 
 (\alpha_2 \alpha_4+\alpha_4^2)X_{s_4s_2} + 
 (\alpha_2+\alpha_3+2\alpha_4)X_{s_3s_4s_2}\\ 
 &+
 (\alpha_2+2\alpha_4+\alpha_5) X_{s_5s_4s_2}+ X_{s_1s_3s_4s_2}+
 2 X_{s_3s_5s_4s_2} + X_{s_6s_5s_4s_2}.
\end{align*}
\end{ex}

\end{document}